\documentclass[11pt]{article}

\usepackage[margin=1.15in]{geometry}
\usepackage{amsmath,amssymb,amsthm,mathtools}
\usepackage{enumitem}
\usepackage[colorlinks=true,citecolor=blue,linkcolor=blue,urlcolor=blue]{hyperref}
\usepackage{microtype}

\newtheorem{theorem}{Theorem}[section]
\newtheorem{lemma}[theorem]{Lemma}

\theoremstyle{definition}

\newtheorem{remark}[theorem]{Remark}

\DeclareMathOperator{\Vol}{Vol}

\DeclareMathOperator{\Hom}{Hom}

\newcommand{\RR}{\mathbb{R}}
\newcommand{\ZZ}{\mathbb{Z}}

\newcommand{\PP}{\mathbb{P}}
\newcommand{\EE}{\mathbb{E}}

\newcommand{\eps}{\varepsilon}

\newcommand{\calE}{\mathcal{E}}

\newcommand{\dd}{\,\mathrm{d}}
\newcommand{\hbin}{h_{\rm bin}}
\newcommand{\qbinom}[2]{\genfrac{[}{]}{0pt}{}{#1}{#2}_q}

\title{Lipschitz Functions on Sparse Graphs II}
\author{Samuel Korsky}
\date{\today}

\begin{document}
\maketitle

\begin{abstract}
\noindent
Korsky, Saffat and Aiylam introduced a growth constant $c(G)$ for integer-valued $h$-Lipschitz functions on a finite graph $G$ and proved that, for $G=G(n,d/n)$,
\[
  \frac{1}{2d}+O(d^{-2})\le \log c(G)\le
  \frac{4\log^2 d}{d}+O(d^{-1})
\]
with high probability.  We sharpen the random-graph part of their result; as $n\to\infty$ and then $d\to\infty$, we prove
\[
  \log c(G)=\frac{\pi^2}{6d}+o(d^{-1})
\]
with high probability.  Additionally, we derive bounds on $\log c(Q_d)$ where $Q_d$ is the $d$-dimensional hypercube graph:
\[
  \frac{\pi^2}{6d}+o(d^{-1}) \le \log{c(Q_d)}\le
  \left(\frac{3}{4} + o(1)\right)\frac{\log d}{d}.
\]
\end{abstract}

\section{Introduction}

Let $G=(V,E)$ be a finite graph and let $h\ge1$ be an integer.  An integer-valued function $f:V\to\ZZ$ is called $h$-Lipschitz if
\[
  |f(u)-f(v)|\le h\qquad\text{for every }uv\in E.
\]
If $G$ has connected components $C_1,\ldots,C_k$, fix one root $r_i\in C_i$ in each component and require $f(r_i)=0$.  Let $N_G(h)$ be the number of such rooted functions.  Korsky, Saffat and Aiylam observed that $N_G(h)$ is an Ehrhart polynomial of degree $|V|-k$ and defined $c(G)$ by
\begin{equation}\label{eq:def-c}
  c(G)^{|V|-k}:=\lim_{h\to\infty}\frac{N_G(h)}{h^{|V|-k}}.
\end{equation}
Equivalently, $c(G)^{|V|-k}$ is the volume of the rooted Lipschitz polytope
\begin{equation}\label{eq:polytope}
  P_G:=\{x\in\RR^V:x(r_i)=0\text{ for all }i,
  \ |x(u)-x(v)|\le1\text{ for all }uv\in E\}.
\end{equation}
This note concerns the random graph $G(n,d/n)$ in the regime where $n\to\infty$ first and $d\to\infty$ afterwards.  The previous bounds from \cite{KorskySaffatAiylam} are
\begin{equation}\label{eq:old-bounds}
  \frac{1}{2d}+O(d^{-2})\le \log c(G(n,d/n))\le
  \frac{4\log^2 d}{d}+O(d^{-1})
\end{equation}
with high probability.  We prove the following first-order asymptotic.

\begin{theorem}\label{thm:main}
For every $\eps>0$ there is $d_0=d_0(\eps)$ such that, for every $d\ge d_0$,
\[
  \lim_{n\to\infty}\PP\left(
  \left|\log c(G(n,d/n))-\frac{\pi^2}{6d}\right|\le\frac{\eps}{d}
  \right)=1.
\]
Equivalently,
\[
  c(G(n,d/n))=1+\frac{\pi^2}{6d}+o_d(d^{-1})
\]
with high probability.
\end{theorem}

\smallskip
\noindent
The lower bound uses a logistic profile which is essentially uniform on an interval of length one and has boundary layers of width $d^{-1}$.  For the random graph, we average first over the independent graph edges and then use bounded differences to pass from the annealed volume to high probability.  The gain in profile entropy is $\pi^2/(3d)$, while the edge-forbidden cost is $\pi^2/(6d)$, leaving $\pi^2/(6d)$.

\bigskip
\noindent
The upper bound is more delicate but has a short final form.  The flatness lemma from \cite{KorskySaffatAiylam} implies that every Lipschitz function has almost all vertices in some interval of length one.  After anchoring that interval as $[0,1]$, the first upper tail is $(1,2]$ and the first lower tail is $[-1,0)$.  For fixed tail sizes $r,s$, the expected tail volume is bounded by
\[
  \frac{q^{rs+r}}{(q;q)_r(q;q)_s},
  \qquad q=1-d/n,
\]
up to $\exp(o_d(n/d))$ and harmless polynomial factors from the choice of the anchor.  Since $q^r\le1$, summing over $r,s$ and using the $q$-binomial theorem gives one $q$-Pochhammer product,
\[
  \sum_{r,s}\frac{q^{rs+r}}{(q;q)_r(q;q)_s}
  \le \sum_{r,s}\frac{q^{rs}}{(q;q)_r(q;q)_s}
  \le \exp(o_d(n/d))\cdot\frac1{(q;q)_\infty}.
\]
Finally,
\[
  -\log{(q;q)_\infty}
  =-\frac nd\int_0^\infty\log{(1-e^{-t})}\,dt+o(n/d)
  =\frac{\pi^2}{6d}\cdot n+o(n/d).
\]
This is the source of the sharp constant.

\bigskip
\noindent
The same logistic profile also gives the lower bound for the hypercube, but the proof is separate: it integrates over one bipartition class and then uses the available intervals on the other class.  The upper bound for the hypercube is different; it uses the homomorphism comparison theorem of Galvin and Tetali \cite{GalvinTetali}.  This yields an auxiliary result,
\[
  \frac{\pi^2}{6d}+o(d^{-1})\le \log c(Q_d)
  \le \left(\frac34+o(1)\right)\frac{\log d}{d}.
\]

\bigskip
\noindent
Typical flatness of Lipschitz functions on expanders was studied by Peled, Samotij and Yehudayoff \cite{PeledSamotijYehudayoff}.  Krueger, Li and Park recently extended this line of work to weak expanders and larger Lipschitz constants \cite{KruegerLiPark}.  The present paper is different in emphasis: it studies the leading Ehrhart-volume growth constant on sparse random graphs.

\bigskip
\noindent
Throughout, $o_d(1)$ denotes a quantity tending to zero as $d\to\infty$, after the limit $n\to\infty$ has been taken.  We write $p=d/n$ and $q=1-p$.

\section{Preliminaries}

\subsection{Volume Normalization}

For a connected graph $H$ and a root $r\in V(H)$, set
\[
  P_H(r):=\{x\in\RR^{V(H)}:x_r=0,
  \ |x(u)-x(v)|\le1\text{ for all }uv\in E(H)\}.
\]
The volume of $P_H(r)$ is independent of the chosen root, since changing the root amounts to changing coordinates in the quotient space $\RR^{V(H)}/\langle\mathbf 1\rangle$ by a determinant-one linear map.  Hence, for connected $H$,
\[
  \Vol(P_H(r))=c(H)^{|V(H)|-1}.
\]
For a disconnected graph, the total rooted volume is the product over connected components.

\bigskip
\noindent
We shall use the following standard facts about $G(n,d/n)$: for fixed $d>1$, with high probability there is a unique giant component $C_1$; its complement has $(\varrho_d+o(1))n$ vertices, where $\varrho_d$ is the solution of $\varrho=e^{-d(1-\varrho)}$; in particular $|V\setminus C_1|=e^{-d+o_d(d)}n$ as $d\to\infty$.  We also use $|E(G)|=(d/2+o(1))n$ with high probability.  See, for example, \cite[Chapters~5--6]{JansonLuczakRucinski}.

\subsection{Two Analytic Estimates}

The following elementary integral occurs in both the lower and upper bounds.

\begin{lemma}\label{lem:zeta-integral}
\[
  -\int_0^\infty \log{(1-e^{-t})}\,dt=\frac{\pi^2}{6}.
\]
\end{lemma}

\begin{proof}
Expand $\log(1-e^{-t})=-\sum_{k\ge1}e^{-kt}/k$ and integrate term by term.
\end{proof}

\smallskip
\noindent
We also need the asymptotics of the $q$-Pochhammer product.  Write
\[
  (q;q)_r:=\prod_{j=1}^r(1-q^j),\qquad
  (q;q)_\infty:=\prod_{j=1}^\infty(1-q^j).
\]

\begin{lemma}\label{lem:qpoch-asymp}
Let $q=1-d/n$, where $d$ is fixed and $n\to\infty$.  Then
\[
  -\log{(q;q)_\infty}
  =\frac{\pi^2}{6d}\cdot n+o(n/d).
\]
The error is uniform for all $d$ in any range with $d=o(n)$.
\end{lemma}

\begin{proof}
Put \(p=d/n\), \(q=1-p\), and \(\lambda=-\log q\). Then \(\lambda=p+O(p^2)\), and \(q^j=e^{-\lambda j}\). Hence
\[
  -\log(q;q)_\infty=\sum_{j=1}^{\infty}-\log(1-q^j)=\sum_{j=1}^{\infty}f(\lambda j),
\]
where \(f(t):=-\log(1-e^{-t})\). The function \(f\) is positive and decreasing on \((0,\infty)\), so the elementary integral comparison
\[
  \int_{\lambda}^{\infty}f(t)\,dt\le \lambda\sum_{j=1}^{\infty}f(\lambda j)\le \int_0^\infty f(t)\,dt
\]
gives
\[
  \sum_{j=1}^{\infty}f(\lambda j)=\frac1\lambda\int_0^\infty f(t)\,dt+O\left(\frac1\lambda\int_0^\lambda f(t)\,dt\right).
\]
Near zero, \(f(t)=-\log t+O(t)\), and therefore
\[
  \int_0^\lambda f(t)\,dt=O(\lambda\log(1/\lambda)).
\]
Thus
\[
  \sum_{j=1}^{\infty}f(\lambda j)=\frac1\lambda\int_0^\infty f(t)\,dt+O(\log(1/\lambda)).
\]
By Lemma~\ref{lem:zeta-integral},
\[
  \int_0^\infty f(t)\,dt=\frac{\pi^2}{6}.
\]
Since \(\lambda=p+O(p^2)\), we have \(1/\lambda=1/p+O(1)\), and hence
\[
  -\log(q;q)_\infty=\frac{\pi^2}{6p}+O(\log(1/p)).
\]
Finally \(p=d/n\), so \(1/p=n/d\), and \(O(\log(1/p))=o(n/d)\). Therefore
\[
  -\log(q;q)_\infty=\frac{\pi^2}{6d}\cdot n +o(n/d),
\]
as claimed.
\end{proof}

\subsection{Concentration Estimates}

We shall use the following standard bounded-differences inequalities in two places.

\begin{lemma}[Bounded differences]\label{lem:bounded-differences}
Let $Y_1,\ldots,Y_m$ be independent random variables, and let $F=F(Y_1,\ldots,Y_m)$ be a real-valued function.  Suppose that changing only the $i$-th coordinate can change $F$ by at most $c_i$.  Then, for every $t>0$,
\[
  \PP(|F-\EE \left[F\right]|\ge t)\le 2\exp\left(-\frac{2t^2}{\sum_i c_i^2}\right).
\]
Moreover,
\[
  \log \EE \left[e^F\right]\le \EE \left[F\right]+\frac18\sum_i c_i^2.
\]
\end{lemma}

\begin{proof}
Let $M_i=\EE(F\mid Y_1,\ldots,Y_i)$ be the Doob martingale.  Its martingale differences satisfy $|M_i-M_{i-1}|\le c_i$.  The two tail estimates are the Azuma--Hoeffding, or McDiarmid, bounded-differences inequality.  For the exponential-moment estimate, Hoeffding's lemma gives
\[
  \EE\left[e^{M_i-M_{i-1}}\mid Y_1,\ldots,Y_{i-1}\right]
  \le \exp(c_i^2/8).
\]
Multiplying these conditional estimates over $i$ gives
\[
  \EE \left[e^{F-\EE \left[F\right]}\right] \le \exp\left(\frac18\sum_i c_i^2\right),
\]
which is the stated bound.
\end{proof}

\section{The Lower Bound}

The lower bound uses the density
\begin{equation}\label{eq:rho-d}
  \rho_d(x):=\frac{1-e^{-d}}{(1+e^{-dx})(1+e^{d(x-1)})},
  \qquad x\in\RR.
\end{equation}
Its distribution function is
\begin{equation}\label{eq:rho-cdf}
  F_d(x)=\int_{-\infty}^x\rho_d(t)\,dt
  =\frac1d\cdot\log\left(\frac{1+e^{dx}}{1+e^{d(x-1)}}\right).
\end{equation}
In particular, $\int_\RR\rho_d(x)\,dx=1$.  The density is close to $1$ on $[0,1]$, has a logistic left boundary layer near $0$, and has a logistic right boundary layer near $1$.

\bigskip
\noindent
For a probability density $\rho$ on $\RR$, define
\begin{equation}\label{eq:H-Q}
  H(\rho):=-\int_\RR\rho(x)\log\rho(x)\,dx,
  \qquad
  Q(\rho):=\iint_{|x-y|>1}\rho(x)\rho(y)\,dx\,dy.
\end{equation}
Here $H(\rho)$ is the continuous entropy of the profile and $Q(\rho)$ is the probability that a random pair of profile values would violate a $1$-Lipschitz edge constraint.

\begin{lemma}[Logistic boundary layers]\label{lem:logistic}
The density \eqref{eq:rho-d} satisfies
\begin{align}
  H(\rho_d)&=\frac{\pi^2}{3d}+O(d^{-2}),\label{eq:entropy-rho}\\
  Q(\rho_d)&=\frac{\pi^2}{3d^2}+O(d^{-3}).\label{eq:bad-rho}
\end{align}
Consequently,
\begin{equation}\label{eq:profile-gain}
  H(\rho_d)-\frac d2 \cdot Q(\rho_d)=\frac{\pi^2}{6d}+O(d^{-2}).
\end{equation}
\end{lemma}

\begin{proof}
Near the right endpoint, put $x=1+t/d$.  Then
\[
  \rho_d(1+t/d)=\frac{1}{1+e^t}+O(e^{-d/3})
\]
uniformly for $|t|\le d^{2/3}$.  Near the left endpoint, put $x=t/d$.  Then
\[
  \rho_d(t/d)=\frac{1}{1+e^{-t}}+O(e^{-d/3})
\]
in the same range.  Outside the two $d^{-1}$-scale boundary layers, the density is exponentially close either to $1$ or to $0$.  Hence
\[
  H(\rho_d)=-\frac2d\int_{-\infty}^{\infty}
  \left(\frac{1}{1+e^t}\right)\log\left(\frac{1}{1+e^t}\right)\,dt+O(d^{-2}).
\]
With $u=(1+e^t)^{-1}$, the integral is
\[
  -\int_0^1\frac{\log u}{1-u}\,du=\frac{\pi^2}{6},
\]
which proves \eqref{eq:entropy-rho}.

\bigskip
\noindent
For $Q(\rho_d)$, by symmetry it is enough to compute $\PP(X-Y>1)$ and double it.  Put $X=1+t/d$.  By \eqref{eq:rho-cdf}, the left-tail mass below $t/d$ is
\[
  F_d(t/d)=\frac1d\cdot\log\left(\frac{1+e^t}{1+e^{t-d}}\right)
  =\frac1d\cdot\log(1+e^t)+O(d^{-2})
\]
uniformly over the contributing range.  Therefore
\[
  \PP(X-Y>1)
  =\frac1{d^2}\int_{-\infty}^{\infty}
  \frac{\log(1+e^t)}{1+e^t}\,dt+O(d^{-3})
  =\frac{\pi^2}{6d^2}+O(d^{-3}).
\]
Doubling gives \eqref{eq:bad-rho}.
\end{proof}

\smallskip
\noindent
We next give a random-graph lower-bound lemma.  For a compact interval
\[
  I=I_{d,T}:=[-T/d,1+T/d],
\]
write
\[
  Z_I(G):=\Vol\{x\in I^{[n]}: |x_i-x_j|\le1\text{ for every }ij\in E(G)\}.
\]
This is an unrooted volume.  Rooting is handled by slicing at the end.

\begin{lemma}[Random-graph profile lower bound]\label{lem:random-profile-lower}
Let $T=T(d)\to\infty$ with $T^2=o(d)$, and let $I=I_{d,T}$.  For fixed $d$, as $n\to\infty$, with high probability,
\[
  \log Z_I(G(n,d/n))
  \ge
  n\left(H(\rho_d)-\frac d2 \cdot Q(\rho_d)+o_d(d^{-1})\right).
\]
Here $\rho_d$ may be replaced by its truncation to $I$ and renormalization; this changes $H(\rho_d)$ by $o_d(d^{-1})$ and $Q(\rho_d)$ by $o_d(d^{-2})$.
\end{lemma}

\begin{proof}
Let $\rho$ be the truncation and renormalization of $\rho_d$ to $I$, and put
\[
  H=H(\rho),\qquad Q=Q(\rho),\qquad W(x)=-\log\rho(x).
\]
For $x\in I^n$, let
\[
  B(x):=\#\{\{i,j\}: |x_i-x_j|>1\}.
\]
Averaging over the random graph gives
\[
  \EE_G \left[Z_I(G)\right]=\int_{I^n} q^{B(x)}\dd x
  =\EE_\rho\left[\exp\left(\sum_{i=1}^n W(X_i)\right)q^{B(X)}\right],
\]
where $X_1,\ldots,X_n$ are independent with density $\rho$.

\bigskip
\noindent
For each fixed $d$, the law of large numbers gives
\[
  \sum_{i=1}^n W(X_i)=nH+o_d(n/d)
\]
with probability tending to one.  Also,
\[
  \EE \left[B(X)\right]=Q\binom n2.
\]
Since $B(X)$ is a sum of bounded pair functions, all summands indexed by disjoint pairs are independent.  Hence only $O(n^3)$ pairs of summands can have nonzero covariance, and
\[
  \operatorname{Var} B(X)=O_d(n^3).
\]
Chebyshev's inequality therefore gives
\[
  B(X)=Q\binom n2+o_d(n^2/d^2)
\]
with probability tending to one, because $n^{3/2}=o(n^2/d^2)$ for each fixed $d$.  Since $\log q=-d/n+O_d(n^{-2})$, on the intersection of these two events we have
\[
  q^{B(X)}
  \ge
  \exp\left(-\frac dn\cdot Q\binom n2-o_d(n/d)\right)
  =\exp\left(-\frac d2 \cdot Qn-o_d(n/d)\right).
\]
Since the intersection of the two typical events has probability tending to one, its logarithm is $o_d(n/d)$.  Therefore
\begin{equation}\label{eq:EZ-random-lower}
  \log \EE_G \left[Z_I(G)\right]
  \ge
  nH-\frac d2\cdot Qn-o_d(n/d).
\end{equation}

\bigskip
\noindent
It remains to pass from the mean to high probability.  Generate $G(n,d/n)$ by independent vertex-exposure blocks, where the block at vertex $i$ consists of the edges from $i$ to earlier vertices.  Let $F(G):=\log Z_I(G)$.  We claim that changing one exposure block changes $F$ by at most $C T/d$.

\bigskip
\noindent
Indeed, suppose that only the block at $i$ is changed.  Fix all coordinates except $x_i$.  The constraints not involving $i$ are unchanged.  The constraints involving $i$ restrict $x_i$ to
\[
  I\cap\bigcap_{j\in J}[x_j-1,x_j+1]
\]
for some set $J$ of neighbors, including the fixed later-neighbor constraints and whichever earlier-neighbor constraints are present.  Since every $x_j\in I=[-T/d,1+T/d]$, this interval has length between $1-2T/d$ and $1+2T/d$.  Hence the two possible fiber lengths, before and after the block change, differ pointwise by a factor at most
\[
  \frac{1+2T/d}{1-2T/d}.
\]
Integrating over the other coordinates gives
\[
  |F(G)-F(G')|
  \le \log\left(\frac{1+2T/d}{1-2T/d}\right)
  \le C\cdot\frac{T}{d}
\]
for all sufficiently large $d$.

\bigskip
\noindent
Applying Lemma~\ref{lem:bounded-differences} with $n$ exposure blocks and $c_i=CT/d$ gives, for every $t>0$,
\[
  \PP\left(|F(G)-\EE_G \left[F(G)\right]|\ge t\right)
  \le 2\exp\left(-\frac{2t^2}{nC^2T^2/d^2}\right).
\]
Taking $t=\eta n/d$ and then letting $n\to\infty$ for each fixed $d$ gives
\[
  F(G)=\EE_G \left[F(G)\right]+o_p(n/d).
\]
The exponential-moment part of Lemma~\ref{lem:bounded-differences} gives
\[
  \log \EE_G \left[e^{F(G)}\right]
  \le
  \EE_G \left[F(G)\right]+O\left(\frac{nT^2}{d^2}\right).
\]
Since $e^{F(G)}=Z_I(G)$ and $T^2=o(d)$, the error term is $o_d(n/d)$.  Combining this with \eqref{eq:EZ-random-lower} yields
\[
  \EE_G \left[F(G)\right]
  \ge
  nH-\frac d2\cdot Qn-o_d(n/d).
\]
Using the concentration estimate for $F(G)$ once more gives the claimed high-probability lower bound.
\end{proof}

\begin{theorem}[Random-graph lower bound]\label{thm:lower}
For every $\eps>0$ there is $d_0=d_0(\eps)$ such that, for every $d\ge d_0$,
\[
  \lim_{n\to\infty}\PP\left(
  \log c(G(n,d/n))\ge\frac{\pi^2/6-\eps}{d}
  \right)=1.
\]
\end{theorem}

\begin{proof}
Choose \(T=T(d)\to\infty\) with \(T^2=o(d)\), for instance \(T=\log d\), and let \(I=I_{d,T}\).  By Lemma~\ref{lem:random-profile-lower} and Lemma~\ref{lem:logistic}, with high probability,
\[
  Z_I(G)\ge
  \exp\left(n\left(\frac{\pi^2}{6d}+o_d(d^{-1})\right)\right).
\]
This is an unrooted volume of valid assignments lying in \(I^n\).

\bigskip
\noindent
Let \(k\) be the number of connected components of \(G\).  With high probability, \(k\le 1+|V\setminus C_1|=e^{-d+o_d(d)}n+1\).  Slicing the unrooted set by the root coordinate in each component loses at most a factor \(|I|^k\).  Since
\[
  \log |I|^k
  =k\log(1+2T/d)=o_d(n/d),
\]
some rooted slice has volume at least
\[
  \exp\left(n\left(\frac{\pi^2}{6d}+o_d(d^{-1})\right)\right).
\]
Translating each component separately sends that slice into the rooted Lipschitz polytope.  Finally, \(|V|-k=n-o_d(n/d)\) on the same exponential scale, so \eqref{eq:def-c} gives the claimed lower bound.
\end{proof}

\section{The Upper Bound}

\subsection{Anchored Flatness}

The upper bound begins with the expansion input already used in \cite{KorskySaffatAiylam}.  We include the short proof for completeness.

\begin{lemma}[Flatness]\label{lem:flatness}
There is an absolute constant $K$ such that the following holds.  Put
\[
  \alpha=\frac{K\log d}{d},\qquad R=\lceil 2\alpha n\rceil.
\]
With high probability, every $1$-Lipschitz function $x$ on the giant component $C_1$ of $G(n,d/n)$ has a vertex $v\in C_1$ such that, after subtracting $x_v$, all but at most $R$ vertices of $C_1$ have values in $[0,1]$.
\end{lemma}

\begin{proof}
Let \(m=\lceil\alpha n\rceil\).  Let \(\calE_\alpha\) be the event that every two disjoint subsets \(A,B\subseteq[n]\) with \(|A|=|B|=m\) have at least one edge between them.  For fixed such \(A,B\),
\[
  \PP(E(A,B)=\varnothing)=(1-d/n)^{m^2}
  \le \exp(-d\alpha^2 n+O(d\alpha)).
\]
The number of possible pairs is at most
\[
  \binom{n}{m}^2\le \exp(2\hbin(\alpha)n+o(n)),
\]
where \(\hbin\) is the binary entropy function.  Since
\[
  \hbin(\alpha)=\alpha\log(1/\alpha)+O(\alpha)
  =K\cdot\left(\frac{\log^2 d+O(\log d\log\log d)}{d}\right)
\]
and
\[
  d\alpha^2=K^2\cdot\frac{\log^2 d}{d},
\]
choosing \(K\) sufficiently large gives \(\PP(\calE_\alpha)\to1\).

\bigskip
\noindent
Assume \(\calE_\alpha\).  Let \(x:C_1\to\RR\) be \(1\)-Lipschitz.  Let \(a\) be the \(m\)-th smallest value of \(x\) on \(C_1\), and choose \(v\in C_1\) with \(x_v=a\).  Then at least \(m\) vertices have value at most \(a\), and at most \(m-1\) vertices have value below \(a\).  If at least \(m\) vertices had value greater than \(a+1\), then the set of vertices with value at most \(a\) and the set of vertices with value greater than \(a+1\) would be two disjoint sets of size at least \(m\) with no edge between them, contradicting \(\calE_\alpha\).  Hence at most \(m-1\) vertices have value greater than \(a+1\).  Therefore all but at most \(2m-2\le R\) vertices lie in \([a,a+1]\).
\end{proof}

\subsection{The One-Tail Integral}

The main enumeration step is a Gaussian-binomial identity.  For $0<q<1$ and integers $N,r\ge0$, define
\[
  A_{N,r}(q):=\int_{[0,1]^N}\int_{[0,1]^r}
  q^{\#\{(i,u):y_u>t_i\}}\,dy\,dt.
\]
The Gaussian binomial coefficient is
\[
  \qbinom{N+r}{r}:=\prod_{j=1}^r\frac{1-q^{N+j}}{1-q^j}.
\]

\begin{lemma}[One-tail $q$-binomial identity]\label{lem:qbinom-integral}
For all $N,r\ge0$,
\begin{equation}\label{eq:A-N-r}
  A_{N,r}(q)=\binom{N+r}{r}^{-1}\qbinom{N+r}{r}.
\end{equation}
Consequently,
\begin{equation}\label{eq:A-upper}
  A_{N,r}(q)
  \le \binom{N+r}{r}^{-1}\prod_{j=1}^r\frac1{1-q^j}.
\end{equation}
\end{lemma}

\begin{proof}
Order the $N+r$ real variables $t_1,\ldots,t_N,y_1,\ldots,y_r$.  Each relative order is equally likely under Lebesgue measure on $[0,1]^{N+r}$.  The statistic $\#\{(i,u):y_u>t_i\}$ is the inversion number of the corresponding binary word with $N$ letters of one type and $r$ letters of the other type.  The generating function of this inversion statistic is the Gaussian binomial coefficient; see, for example, \cite[Chapter~1]{AndrewsPartitions} or \cite[Section~1.7]{StanleyEC1}.  Dividing by the number $\binom{N+r}{r}$ of binary words gives \eqref{eq:A-N-r}.  The inequality follows from $1-q^{N+j}\le1$.
\end{proof}

\smallskip
\noindent
The following correlation observation lets the two tails decouple before the cross-tail edge penalty is applied.

\begin{lemma}[Opposite monotonicities]\label{lem:monotone}
Fix $N,r,s\ge0$ and $0<q<1$.  For $t\in[0,1]^N$, set
\[
  B_r(t):=\int_{[0,1]^r}q^{\#\{(i,u):y_u>t_i\}}\,dy,
\]
and
\[
  C_s(t):=\int_{[0,1]^s}q^{\#\{(i,w):z_w>1-t_i\}}\,dz.
\]
Then
\[
  \int_{[0,1]^N}B_r(t)C_s(t)\,dt
  \le A_{N,r}(q)A_{N,s}(q).
\]
\end{lemma}

\begin{proof}
The function $B_r(t)$ is increasing in each coordinate $t_i$, while $C_s(t)$ is decreasing in each coordinate.  Under product Lebesgue measure on $[0,1]^N$, increasing and decreasing functions are negatively correlated.  Applying this elementary form of the FKG/Chebyshev correlation inequality gives the claim.
\end{proof}

\subsection{Summing the Two Tails}

We now prove the summation bound which produces the constant.  Recall that
\[
  (q;q)_r=\prod_{j=1}^r(1-q^j).
\]

\begin{lemma}[$q$-Pochhammer summation]\label{lem:qpoch-sum}
Let $q=1-d/n$ and let $R=O(n\log d/d)$.  Then
\begin{equation}\label{eq:qpoch-sum}
  \sum_{0\le r,s\le R}\frac{q^{rs}}{(q;q)_r(q;q)_s}
  \le \exp(o_d(n/d))\cdot\frac1{(q;q)_\infty}.
\end{equation}
Consequently,
\begin{equation}\label{eq:qpoch-sum-pi}
  \sum_{0\le r,s\le R}\frac{q^{rs}}{(q;q)_r(q;q)_s}
  \le \exp\left(\left(\frac{\pi^2}{6}+o_d(1)\right)\cdot\frac nd\right).
\end{equation}
\end{lemma}

\begin{proof}
We use the $q$-binomial theorem in the form
\begin{equation}\label{eq:q-binomial-theorem}
  \sum_{s=0}^\infty\frac{z^s}{(q;q)_s}=\frac1{(z;q)_\infty},
  \qquad |z|<1.
\end{equation}
For $r\ge1$, take $z=q^r$.  Then
\[
  \sum_{s=0}^\infty\frac{q^{rs}}{(q;q)_s}
  =\frac1{(q^r;q)_\infty}.
\]
Therefore
\[
  \sum_{s=0}^\infty\frac{q^{rs}}{(q;q)_r(q;q)_s}
  =\frac1{(q;q)_r(q^r;q)_\infty}
  =\frac1{(q;q)_\infty(1-q^r)}.
\]
Since $1-q^r\ge 1-q=d/n$, summing over $1\le r\le R$ gives
\[
  \sum_{1\le r\le R}\sum_{s=0}^R\frac{q^{rs}}{(q;q)_r(q;q)_s}
  \le \frac{R}{1-q}\cdot\frac1{(q;q)_\infty}.
\]
The row $r=0$ is bounded by
\[
  \sum_{s=0}^R\frac1{(q;q)_s}\le \frac{R+1}{(q;q)_\infty}.
\]
Since $R=O(n\log d/d)$, the logarithm of the prefactor is $o(n/d)$.  This proves \eqref{eq:qpoch-sum}.  Equation \eqref{eq:qpoch-sum-pi} follows from Lemma~\ref{lem:qpoch-asymp}.
\end{proof}

\subsection{The Anchored Volume Bound}

We next define the random variable whose expectation will bound the volume of all anchored flat functions.  Fix an anchor $a\in[n]$.  For a vector $x\in\RR^n$ with $x_a=0$, partition the vertices into
\begin{align*}
  S(x)&:=\{i:x_i\in[0,1]\},\\
  U(x)&:=\{i:x_i\in(1,2]\},\\
  W(x)&:=\{i:x_i\in[-1,0)\},\\
  D(x)&:=\{i:x_i>2\text{ or }x_i<-1\}.
\end{align*}
We shall only count vectors with $a\in S(x)$ and $|U(x)|+|W(x)|+|D(x)|\le R$.

\bigskip
\noindent
A vector arising from a Lipschitz function on the giant component has the following additional property: every vertex of $D(x)$ is connected, through vertices outside $S(x)$, to $U(x)\cup W(x)$.  Indeed, no edge from $D(x)$ to $S(x)$ can be Lipschitz, so a path from a deep vertex to the central set must first pass through the first upper or lower layer.  To overcount such vectors, we allow a directed forest: every vertex of $D$ chooses a parent in $D\cup U\cup W$, and repeated parent maps eventually reach $U\cup W$.  There are at most $n^{|D|}$ such parent maps.

\bigskip
\noindent
Let $Z_R(G)$ be the sum, over anchors $a$, over all choices of disjoint sets $S,U,W,D$ with $a\in S$ and $|U|+|W|+|D|\le R$, and over all such directed forests from $D$ to $U\cup W$, of the volume of vectors satisfying:
\begin{enumerate}[label=(\roman*)]
\item $x_a=0$;
\item $x_i\in[0,1]$ for $i\in S$, $x_i\in(1,2]$ for $i\in U$, and $x_i\in[-1,0)$ for $i\in W$;
\item for every forest edge $ij$, $|x_i-x_j|\le1$ and the graph edge $ij$ is present;
\item every non-forest graph edge $ij$ with $|x_i-x_j|>1$ is absent.
\end{enumerate}
The definition overcounts the actual anchored Lipschitz volume, which is useful for an upper bound.

\begin{lemma}[Expected anchored flat volume]\label{lem:expected-Z}
Let $R=O(n\log d/d)$.  Then
\begin{equation}\label{eq:expected-Z}
  \EE \left[Z_R(G(n,d/n))\right]
  \le \exp\left(\left(\frac{\pi^2}{6}+o_d(1)\right)\frac nd\right).
\end{equation}
\end{lemma}

\begin{proof}
Fix an anchor $a$ and fixed disjoint sets $S,U,W,D$ with $a\in S$.  Write
\[
  N=|S|,
  \qquad r=|U|,
  \qquad s=|W|,
  \qquad m=|D|.
\]
First take $m=0$.  For $i\in S\setminus\{a\}$, write $x_i=t_i\in[0,1]$; for $u\in U$, write $x_u=1+y_u$ with $y_u\in[0,1]$; for $w\in W$, write $x_w=-z_w$ with $z_w\in[0,1]$.

\bigskip
\noindent
An edge $ui$ with $u\in U$ and $i\in S\setminus\{a\}$ is forbidden precisely when $y_u>t_i$.  An edge $wi$ with $w\in W$ and $i\in S\setminus\{a\}$ is forbidden precisely when $z_w>1-t_i$.  Every edge between $U$ and $W$ is forbidden, except on a measure-zero boundary, because
\[
  (1+y_u)-(-z_w)=1+y_u+z_w>1.
\]
The anchor contributes asymmetrically: every edge from $a$ to $U$ is forbidden, except on a measure-zero boundary, while every edge from $a$ to $W$ is compatible.  Therefore the expected volume for these fixed sets is at most
\[
  q^{rs+r}\int_{[0,1]^{N-1}}\int_{[0,1]^r}\int_{[0,1]^s}
  q^{\#\{(i,u):y_u>t_i\}+\#\{(i,w):z_w>1-t_i\}}
  \dd z\,\dd y\,\dd t.
\]
By Lemma~\ref{lem:monotone} and Lemma~\ref{lem:qbinom-integral}, the integral is at most
\begin{equation}\label{eq:fixed-nondeep}
  q^{rs+r}
  \binom{N-1+r}{r}^{-1}\binom{N-1+s}{s}^{-1}
  \prod_{j=1}^r\frac1{1-q^j}
  \prod_{j=1}^s\frac1{1-q^j}.
\end{equation}

\smallskip
\noindent
Now sum over choices of the anchor and over choices of $S,U,W$ with $m=0$ and fixed $r,s$.  The number of choices after the anchor is fixed is
\[
  \frac{(n-1)!}{(N-1)!r!s!}.
\]
Multiplying by the two binomial denominators in \eqref{eq:fixed-nondeep} and then summing over the $n$ possible anchors leaves the factor
\[
  n\cdot\frac{(n-1)!}{(N-1)!r!s!}
  \binom{N-1+r}{r}^{-1}\binom{N-1+s}{s}^{-1}
  =\frac{n!(N-1)!}{(N+r-1)!(N+s-1)!}.
\]
Since $r,s\le R=O(n\log d/d)$, Stirling's formula gives
\[
  \log\left(\frac{n!(N-1)!}{(N+r-1)!(N+s-1)!}\right)
  =O\left(\log n+\frac{(r+s)^2}{n}\right)=o_d(n/d).
\]
Thus the total contribution with $m=0$ is at most
\begin{equation}\label{eq:m0-bound}
  \exp(o_d(n/d))
  \sum_{0\le r,s\le R}\frac{q^{rs+r}}{(q;q)_r(q;q)_s}
  \le
  \exp(o_d(n/d))
  \sum_{0\le r,s\le R}\frac{q^{rs}}{(q;q)_r(q;q)_s}.
\end{equation}
By Lemma~\ref{lem:qpoch-sum}, this is at most the right-hand side of \eqref{eq:expected-Z}.

\bigskip
\noindent
It remains to check that deep vertices do not change the exponential rate.  Fix $D$ with $|D|=m$.  Since each vertex in $D$ is at distance more than one from every central vertex in $S$, all edges between $D$ and $S$ must be absent.  This gives a factor
\[
  q^{mN}\le \exp(-(1-o_d(1))dm),
\]
because $N\ge n-R-m\ge(1-o_d(1))n$ for the terms under consideration.  The directed forest from $D$ to $U\cup W$ has at most $n^m$ choices and its edges are present with probability $p^m$, giving a factor at most $d^m$.  Given the non-deep coordinates and the forest, the deep coordinates are parametrized by the forest edge increments, each lying in an interval of length at most $2$; hence their volume is at most $2^m$.

\bigskip
\noindent
After also choosing the set $D$, the total deep multiplier is bounded by
\begin{equation}\label{eq:deep-sum}
  \sum_{m\le R}\binom{n}{m}(2d)^m\exp(-(1-o_d(1))dm).
\end{equation}
Let
\[
  a_d=(2d)\exp(-(1-o_d(1))d).
\]
Then \eqref{eq:deep-sum} is at most
\[
  \sum_{m=0}^n\binom{n}{m}a_d^m=(1+a_d)^n.
\]
Thus its logarithm is at most
\[
  na_d=n\exp(-d+O(\log d))=o_d(n/d).
\]
Multiplying this by the non-deep estimate \eqref{eq:m0-bound} proves \eqref{eq:expected-Z}.
\end{proof}

\subsection{Completing the Upper Bound}

\begin{theorem}[Upper bound]\label{thm:upper}
For every $\eps>0$ there is $d_0=d_0(\eps)$ such that, for every $d\ge d_0$,
\[
  \lim_{n\to\infty}\PP\left(
  \log c(G(n,d/n))\le\frac{\pi^2/6+\eps}{d}
  \right)=1.
\]
\end{theorem}

\begin{proof}
Let $C_1$ be the giant component.  On the flatness event of Lemma~\ref{lem:flatness}, every $1$-Lipschitz function on $C_1$ is covered by one of the anchored classes counted by $Z_R(G)$: choose the anchor $v$ supplied by Lemma~\ref{lem:flatness}, subtract $x_v$, and place all vertices outside $C_1$ arbitrarily in the central interval $[0,1]$, which has volume one in each coordinate.  Changing from a fixed root to the chosen anchor is a determinant-one coordinate change in the quotient by constant functions.  Thus the rooted volume of $C_1$ is bounded by the anchored flat volume counted by $Z_R(G)$, and
\begin{equation}\label{eq:vol-C1-Z}
  \Vol(P_{C_1})\le Z_R(G)
\end{equation}
on the flatness event.

\bigskip
\noindent
By Lemma~\ref{lem:expected-Z} and Markov's inequality, with high probability
\[
  Z_R(G)\le
  \exp\left(\left(\frac{\pi^2}{6}+\frac{\eps}{2}\right)\frac nd\right)
\]
for all sufficiently large $d$.  Hence
\[
  \Vol(P_{C_1})\le
  \exp\left(\left(\frac{\pi^2}{6}+\frac{\eps}{2}\right)\frac nd\right)
\]
with high probability.

\bigskip
\noindent
The components outside $C_1$ contain $e^{-d+o_d(d)}n$ vertices with high probability.  Since every connected component on $m$ vertices has rooted Lipschitz volume at most $2^{m-1}$, the total contribution of the non-giant components is
\[
  \exp(o_d(n/d)).
\]
Combining this with \eqref{eq:vol-C1-Z} gives
\[
  \Vol(P_G)\le
  \exp\left(\left(\frac{\pi^2}{6}+\eps\right)\frac nd\right)
\]
with high probability.  Since $|V|-k=n-o_d(n/d)$ in the exponent scale relevant here, \eqref{eq:def-c} gives
\[
  \log c(G(n,d/n))\le\frac{\pi^2/6+\eps}{d}
\]
with high probability.
\end{proof}

\begin{proof}[Proof of Theorem~\ref{thm:main}]
The lower bound is Theorem~\ref{thm:lower} and the upper bound is Theorem~\ref{thm:upper}.  Since $e^x=1+x+O(x^2)$ for $x=O(d^{-1})$, the logarithmic statement is equivalent to
\[
  c(G(n,d/n))=1+\frac{\pi^2}{6d}+o_d(d^{-1}).
\]
\end{proof}

\section{Hypercubes}\label{sec:hypercube}

Let $Q_d$ denote the $d$-dimensional hypercube.  We record a lower bound from the logistic profile and an upper bound from the Galvin--Tetali homomorphism comparison theorem \cite{GalvinTetali}.

\begin{lemma}[Neighbor extreme values]\label{lem:neighbor-extremes}
Let $T=T(d)\to\infty$ with $T^2=o(d)$, let
\[
  I=[-T/d,1+T/d],
\]
and let $\rho$ be the truncation and renormalization of $\rho_d$ to $I$.  Let $X_1,\ldots,X_d$ be independent samples from $\rho$, and set
\[
  M=\max_i X_i,
  \qquad
  m=\min_i X_i,
  \qquad
  R=m+1-M.
\]
Then
\[
  \EE \left[R\right]=o_d(d^{-1}),
  \qquad
  \EE \left[R^2\right]=O(d^{-2}),
\]
and consequently
\[
  \EE\left[\log(2-(M-m))\right]=\EE\left[\log(1+R)\right]\ge -o_d(d^{-1}).
\]
\end{lemma}

\smallskip
\noindent
\begin{proof}
First use the untruncated density $\rho_d$, and let $\widetilde M$ be the maximum of $d$ independent samples.  For fixed $t\in\RR$, the exact distribution function \eqref{eq:rho-cdf} gives
\[
  F_d(1+t/d)
  =1-\frac{\log(1+e^{-t})}{d}+o(d^{-1}).
\]
Hence
\[
  \PP\bigl(d(\widetilde M-1)\le t\bigr)
  =F_d(1+t/d)^d
  \longrightarrow
  \exp(-\log(1+e^{-t}))
  =\frac1{1+e^{-t}}.
\]
The same formula gives exponential tail bounds.  Indeed, for $t\ge0$,
\[
  \PP\bigl(d(\widetilde M-1)>t\bigr)
  \le d\left(1-F_d(1+t/d)\right)
  \le C e^{-t},
\]
and, for $0\le t\le d/2$,
\[
  \PP\bigl(d(\widetilde M-1)<-t\bigr)
  =F_d(1-t/d)^d
  \le C e^{-t},
\]
while the remaining range $t>d/2$ is exponentially small in $d$.  Thus $d(\widetilde M-1)$ converges in $L^2$ to the standard logistic law, which has mean zero and finite second moment.

\bigskip
\noindent
The probability that one of the $d$ samples lies outside $I$ is $O(e^{-T})$, by the same tail estimates.  Since $T\to\infty$, truncating and renormalizing does not change the $L^2$ limit.  Therefore, under the truncated law,
\[
  d(M-1)\to Y\quad\text{in }L^2,
\]
where $Y$ is standard logistic.  The truncated density is symmetric under $x\mapsto 1-x$, so $dm$ has the same distribution as $-d(M-1)$.  Since $\EE \left[Y\right]=0$ and $\EE \left[Y^2\right]<\infty$,
\[
  \EE \left[R\right]=\EE\left[m+1-M\right]=o_d(d^{-1}),
  \qquad
  \EE \left[R^2\right]=O(d^{-2}).
\]
Finally $R\in[-2T/d,1]$, so for large $d$ we have $R\ge-1/2$.  The elementary inequality
\[
  \log(1+r)\ge r-r^2\qquad(-1/2\le r\le1)
\]
gives
\[
  \EE\left[\log(1+R)\right]\ge \EE \left[R\right]-\EE \left[R^2\right]=-o_d(d^{-1}).
\]
\end{proof}

\begin{theorem}[Hypercube bounds]\label{thm:hypercube}
As $d\to\infty$,
\[
  \frac{\pi^2}{6d}+o(d^{-1})\le \log c(Q_d)
  \le \left(\frac34+o(1)\right)\frac{\log d}{d}.
\]
\end{theorem}

\begin{proof}
We first prove the lower bound.  Let $N=2^d$, and write the bipartition of $Q_d$ as $A\cup B$, with $|A|=|B|=N/2$.  By symmetry we may root later at a vertex in $A$.  Choose $T=T(d)\to\infty$ with $T^2=o(d)$, for instance $T=\log d$, and let $\rho$ be the truncation and renormalization of $\rho_d$ to
\[
  I=[-T/d,1+T/d].
\]
For $x\in I^A$ and $b\in B$, define
\[
  \ell_b(x):=\left|\bigcap_{a\sim b}[x_a-1,x_a+1]\right|.
\]
Since all values $x_a$ lie in an interval of length $1+2T/d<2$, the interval above is nonempty and
\[
  1-2T/d\le \ell_b(x)\le 2.
\]
Thus
\[
  Z_Q:=\int_{I^A}\prod_{b\in B}\ell_b(x)\dd x_A
\]
is an unrooted volume of valid Lipschitz assignments on $Q_d$: after choosing the values on $A$, each vertex $b\in B$ may be chosen independently in the interval of length $\ell_b(x)$.

\bigskip
\noindent
Let $X_a$, $a\in A$, be independent with density $\rho$, and put $W=-\log\rho$.  Then
\[
  Z_Q=\EE_\rho\left[\exp\left(\sum_{a\in A}W(X_a)\right)
  \prod_{b\in B}\ell_b(X)\right].
\]
Since $W$ is bounded by $O(T)$ under the truncated law, changing one variable $X_a$ changes $\sum_{a\in A}W(X_a)$ by at most $CT$.  Lemma~\ref{lem:bounded-differences} gives
\[
  \PP\left(\sum_{a\in A}W(X_a)-|A|H(\rho)\le -t\right)
  \le \exp\left(-\frac{2t^2}{(N/2)C^2T^2}\right).
\]
Taking $t=\gamma_d N/d$ with $\gamma_d\downarrow0$ and $\gamma_d^2N/(d^2T^2)\to\infty$ gives
\begin{equation}\label{eq:Q-W-conc}
  \sum_{a\in A}W(X_a)
  \ge |A|H(\rho)-o_d(N/d)
\end{equation}
with probability tending to one.

\bigskip
\noindent
We next control the product of interval lengths.  Fix $b\in B$, and let $M_b$ and $m_b$ be the maximum and minimum of the $d$ neighboring values $X_a$.  Then
\[
  \ell_b=2-(M_b-m_b).
\]
By Lemma~\ref{lem:neighbor-extremes},
\begin{equation}\label{eq:ell-expectation}
  \EE\left[\log\ell_b\right]\ge -o_d(d^{-1}).
\end{equation}
Now set
\[
  S(X):=\sum_{b\in B}\log\ell_b(X).
\]
Linearity of expectation and \eqref{eq:ell-expectation} give
\[
  \EE \left[S\right]\ge -o_d(N/d).
\]
Changing one coordinate $X_a$ affects only the $d$ terms with $b\sim a$.  Each term $\log\ell_b$ lies between $\log(1-2T/d)$ and $\log 2$, so changing one coordinate can change $S$ by at most $C d$.  Applying Lemma~\ref{lem:bounded-differences} with $|A|=N/2$ variables and $c_a=Cd$ gives, for every $t>0$,
\[
  \PP(S-\EE \left[S\right]\le -t)
  \le \exp\left(-\frac{2t^2}{(N/2)C^2d^2}\right).
\]
Choose any $\eta_d\downarrow0$ with $\eta_d^2N/d^4\to\infty$.  Taking $t=\eta_d N/d$ gives
\begin{equation}\label{eq:Q-length-conc}
  S(X)\ge -o_d(N/d)
\end{equation}
with probability tending to one.

\bigskip
\noindent
On the intersection $\calE$ of the high-probability events \eqref{eq:Q-W-conc} and \eqref{eq:Q-length-conc},
\[
  \exp\left(\sum_{a\in A}W(X_a)\right)
  \prod_{b\in B}\ell_b(X)
  \ge
  \exp\left(|A|H(\rho)-o_d(N/d)\right).
\]
Since $\PP_\rho(\calE)\to1$, its logarithm is $o_d(N/d)$.  Therefore
\[
  Z_Q\ge \PP_\rho(\calE)\exp\left(|A|H(\rho)-o_d(N/d)\right)
  =\exp\left(|A|H(\rho)-o_d(N/d)\right).
\]
Since $|A|=N/2$, $H(\rho)=H(\rho_d)+o_d(d^{-1})$, and Lemma~\ref{lem:logistic} gives $H(\rho_d)=\pi^2/(3d)+O(d^{-2})$, we obtain
\[
  Z_Q\ge \exp\left(N\left(\frac{\pi^2}{6d}+o(d^{-1})\right)\right).
\]
Slicing by the coordinate of a root in $A$ loses only a factor $|I|=1+o_d(1)$, so
\[
  \log c(Q_d)\ge\frac{\pi^2}{6d}+o(d^{-1}).
\]

\bigskip
\noindent
We prove the upper bound.  Let $N=2^d$, fix a root $o\in V(Q_d)$, and let $h$ be a positive integer.  Choose a fixed integer $L\ge5$ and set $M=Lh$.  Let $T_{M,h}$ be the graph with vertex set $\ZZ/M\ZZ$, with a loop at every vertex, and with an edge between two vertices when their cyclic distance is at most $h$.

\bigskip
\noindent
We first observe that rooted $h$-Lipschitz functions on $Q_d$ are in bijection with homomorphisms $\varphi:Q_d\to T_{M,h}$ satisfying $\varphi(o)=0$.  Reduction modulo $M$ gives one direction.  Conversely, given such a homomorphism, every oriented edge $uv$ has a unique lift increment $a_{uv}\in[-h,h]\cap\ZZ$ satisfying
\[
  \varphi(v)-\varphi(u)\equiv a_{uv}\pmod M,
\]
because $M>2h$.  Around every $4$-cycle, the sum of the four lifted increments is congruent to $0$ modulo $M$ and has absolute value at most $4h<M$, so it is actually $0$.  Since the cycle space of the hypercube is generated by its $4$-cycles, the lifted increments have zero sum around every cycle.  Integrating them from the root gives a unique rooted integer-valued $h$-Lipschitz function.

\bigskip
\noindent
The target $T_{M,h}$ is vertex-transitive.  Hence the number $N_{Q_d}(h)$ of rooted $h$-Lipschitz functions on $Q_d$ is
\begin{equation}\label{eq:Q-hom-rooted}
  N_{Q_d}(h)=\frac{\Hom(Q_d,T_{M,h})}{M}.
\end{equation}
By the Galvin--Tetali theorem, for every finite $d$-regular bipartite graph $G$ and every finite target graph $H$ with loops allowed,
\[
  \Hom(G,H)\le \Hom(K_{d,d},H)^{|V(G)|/(2d)}.
\]
Applying this with $G=Q_d$ and $H=T_{M,h}$ gives
\begin{equation}\label{eq:GT-Q}
  \Hom(Q_d,T_{M,h})
  \le \Hom(K_{d,d},T_{M,h})^{N/(2d)}.
\end{equation}
The same lifting argument applies to $K_{d,d}$, whose cycle space is also generated by $4$-cycles.  Therefore
\[
  \Hom(K_{d,d},T_{M,h})=M N_{K_{d,d}}(h)
  =L V_d h^{2d}+O_d(h^{2d-1}),
\]
where
\[
  V_d:=\Vol(P_{K_{d,d}})
\]
with one vertex of $K_{d,d}$ rooted.  Combining this with \eqref{eq:Q-hom-rooted} and \eqref{eq:GT-Q}, then taking the leading coefficient of $h^{N-1}$, gives
\begin{equation}\label{eq:Q-coeff-bound}
  c(Q_d)^{N-1}
  \le L^{-1}(L V_d)^{N/(2d)}.
\end{equation}
It remains to estimate $V_d$.

\bigskip
\noindent
Root one vertex in the left part of $K_{d,d}$.  If the values on the left part have range $r\le2$, then every vertex in the right part has an available interval of length $2-r$.  The volume of the set of left-part values with root fixed and range at most $r$ is $d r^{d-1}$; differentiating in $r$ gives
\begin{equation}\label{eq:Kdd-volume}
  V_d=d(d-1)\int_0^2 r^{d-2}(2-r)^d\,dr
  =2^{2d-1}d(d-1)B(d-1,d+1).
\end{equation}
By Stirling's formula,
\[
  V_d\sim \sqrt{\pi}\,d^{3/2}.
\]
Taking logarithms in \eqref{eq:Q-coeff-bound} and using $N=2^d$ yields
\[
  \log c(Q_d)
  \le \frac{\log V_d+O(1)}{2d}+o(1/d)
  =\left(\frac34+o(1)\right)\frac{\log d}{d}.
\]
This proves the theorem.
\end{proof}

\begin{remark}
The hypercube should not be viewed as a random-graph analogue with the same proof.  In the Erdos--Renyi case the sharp upper bound comes from the abundance and near-independence of potential edges between separated level sets.  This mechanism is absent in $Q_d$, whose edges have rigid product structure.  In particular, $Q_d$ admits $1$-Lipschitz functions of large range, so one cannot first prove a global flatness statement of the kind used for random graphs.  The estimates in this section therefore use a different, comparison-based argument.  By contrast, the lower bound uses the logistic profile through the bipartite structure of the cube.
\end{remark}

\section*{Acknowledgements} The author acknowledges GPT-5.5 for producing fully the mechanism of the upper bound for the hypercube graph, and its assistance in performing the detailed computations and preparing an initial draft of this preprint. The proof idea and the direction of the arguments for the sharp random graph bound are due to the author.

\end{document}